\documentclass[12pt,a4paper]{amsart}

\usepackage[utf8]{inputenc}
\usepackage[american]{babel}
\usepackage{amsmath,amssymb,amsthm}
\usepackage{calrsfs}
\usepackage{dsfont}
\usepackage{braket}
\usepackage{csquotes}
\usepackage[left=3cm,right=3cm,top=3cm,bottom=3cm]{geometry}
\usepackage{comment}
\usepackage[hidelinks]{hyperref}
\usepackage{cleveref}

\newcommand{\1}{\mathds{1}}

\newcommand{\tr}{\operatorname{tr}}

\renewcommand{\epsilon}{\varepsilon}
\renewcommand{\phi}{\varphi}

\renewcommand{\AA}{\mathfrak A}
\newcommand{\HH}{\mathfrak H}

\newcommand{\E}{\mathcal E}
\renewcommand{\H}{\mathcal H}
\newcommand{\I}{\mathcal I}
\newcommand{\J}{\mathcal J}
\renewcommand{\L}{\mathcal L}
\newcommand{\U}{\mathcal U}

\newcommand{\IB}{\mathbb B}
\newcommand{\IC}{\mathbb C}
\newcommand{\IN}{\mathbb N}
\newcommand{\IR}{\mathbb R}
\newcommand{\IZ}{\mathbb Z}

\newcommand{\Ad}{\mathrm{Ad}}
\newcommand{\cb}{\mathrm{cb}}
\newcommand{\cp}{\mathrm{cp}}
\newcommand{\dom}{\operatorname{dom}}
\newcommand{\id}{\mathrm{id}}

\newcommand{\abs}[1]{\lvert#1\rvert}
\newcommand{\norm}[1]{\lVert#1\rVert}

\theoremstyle{plain}
\newtheorem{proposition}{Proposition}[section]
\newtheorem{lemma}[proposition]{Lemma}
\newtheorem{theorem}[proposition]{Theorem}
\newtheorem{corollary}[proposition]{Corollary}

\theoremstyle{definition}
\newtheorem{definition}[proposition]{Definition}

\theoremstyle{remark}
\newtheorem{remark}[proposition]{Remark}
\newtheorem{example}[proposition]{Example}

\title[Logarithmic Sobolev Inequalities]{Exponential Relative Entropy Decay Along Quantum Markov Semigroups}
\author{Melchior Wirth}
\address{Institute of Science and Technology, Am Campus 1, 3400 Klosterneuburg, Austria}
\address{Leipzig University, Institute of Mathematics, Neues Augusteum, Augustusplatz 10, 04109 Leipzig, Germany}
\email{melchior.wirth@uni-leipzig.de}

\begin{document}
\begin{abstract}
    We establish the equivalence between exponential decay of the relative entropy along a quantum Markov semigroup and the modified logarithmic Sobolev inequality for general von Neumann algebras. We also extend an intertwining criterion for the modified logarithmic Sobolev inequality to GNS-symmetric quantum Markov semigroups on infinite-dimensional von Neumann algebras.
\end{abstract}

\maketitle

\section*{Introduction}

One of the key properties of the (quantum) relative entropy is the data processing inequality: The relative entropy between two (quantum) states does not increase after application of a (quantum) channel. In particular, the relative entropy between a state and an invariant state of the channel does not increase after iterated application of the channel.

It is of interest in classical probability and analysis as well as various disciplines of mathematical physics such as quantum information theory to obtain bounds that show a strict decrease of the relative entropy with respect to an invariant state for specific channels. Since the relative entropy is an upper bound for the trace distance (Pinsker's inequality), such decay estimates can be used for example to establish mixing time estimates.

In this article, we are concerned with a continuous-time version of this problem in which iterated application of a quantum channel is replaced by a continuous-parameter family of quantum channels that satisfy the semigroup property. Such families are called quantum Markov semigroups. The analog of decay bounds obtained from iterated application of a channel in the continuous-time case are decay bounds that are exponential in time. Such bounds are for example of interest in quantum information theory for their applications to rapid mixing \cite{KT13,Bar17}.

In the classical case \cite{BT03} or for quantum Markov semigroups acting on matrix algebras \cite{KT13,CM15}, it is well known that exponential decay of the relative entropy is equivalent to a logarithmic Sobolev-type inequality, called the modified logarithmic Sobolev inequality or $1$-log Sobolev inequality. If the quantum Markov semigroup has (positive) generator $\mathcal L_\ast$ and unique positive definite invariant state $\sigma$, the modified logarithmic Sobolev inequality takes the form
\begin{equation*}
    \beta \mathrm{tr}(\rho(\log\rho-\log\sigma))\leq \mathrm{tr}(\mathcal L_\ast(\rho)(\log\rho-\log \sigma)).
\end{equation*}
For classical diffusions, this modified logarithmic Sobolev inequality is equivalent to the logarithmic Sobolev inequality \cite{Gro75}, but this equivalence breaks down already for discrete spaces due to a lack of chain rule.

If one moves to infinite dimensions, a mathematically rigorous treatment of the equivalence between the modified logarithmic Sobolev inequality and exponential decay of relative entropy along the semigroup is still lacking (see the remarks in \cite[Section 3.2]{HKV17} for example). Let us briefly discuss the mathematical difficulties in the infinite-dimensional case. The key step in establishing the aforementioned equivalence is the deBruijn identity, which states that the right side of the modified logarithmic Sobolev inequality (called the entropy production or Fisher information) coincides with the negative derivative of the relative entropy along the semigroup.

In infinite dimensions, it is not true that all trajectories of the semigroup are differentiable and the relative entropy is not continuous, let alone differentiable. Moreover, the entropy production in the form $\mathrm{tr}(\L_\ast(\rho)(\log\rho-\log\sigma))$ has the problem that the argument of the trace has no definite sign so that the trace is in general ill-defined. Additionally, if one moves to normal states on a general von Neumann algebra, possibly type III, where normal states cannot be represented by density matrices, one is faced with the challenge of making sense of expressions like $\log\rho-\log\sigma$ in the first place.

In this article, we give a proof of the equivalence between the modified logarithmic Sobolev inequality and exponential relative decay of the relative entropy along a quantum Markov semigroup in a fairly general setting: The quantum Markov semigroup acts on the predual of an arbitrary von Neumann algebra and has a faithful normal invariant state (not necessarily unique). We do not assume that the semigroup has a bounded generator or a generator in generalized GKLS form, and we do not assume any detailed balance or symmetry condition.

Our approach relies on the Haagerup reduction method, which allows us to reduce the problem to finite von Neumann algebras, although the reduction is still technically challenging. To use this method, we first prove the deBruijn identity for finite von Neumann algebras (\Cref{thm:DeBruijn_semifinite}). We then apply the reduction method to the trajectories instead of the full semigroup, which allows us to avoid any compatibility assumptions with the modular group. In this way, we establish the deBruijn identity with an explicit expression of the entropy production in general von Neumann algebras (\Cref{thm:DeBruijn_general}) and then the main result of this article, the equivalence between the modified logarithmic Sobolev inequality and exponential decay of the relative entropy along the semigroup (\Cref{thm:main}).

In the last part, we show (\Cref{thm:intertwining}) how the intertwining approach to modified logarithmic Sobolev inequalities that originates in the work of Carlen and Maas for finite-dimensional von Neumann algebras \cite{CM17,CM20} can be extended to GNS-symmetric quantum Markov semigroups on general von Neumann algebras. As an application, we prove the modified logarithmic Sobolev inequality for quantum Markov semigroups that arise from certain cocycles on groups.

\subsection*{Acknowledgments}
The author is grateful for extended discussions with Martijn Caspers, Matthijs Vernooij and Haonan Zhang, in particular on the topic of the intertwining approach to modified logarithmic Sobolev inequalities. He is grateful to Marius Junge for pointing out the importance of the Haagerup reduction method to him, to Federico Girotti for raising the point of the qualitative long-time behavior of semigroups that satisfy a modified logarithmic Sobolev inequality, and to Stefan Hollands for finding an error in a previous version of the proof of \Cref{thm:main} and pointing the author to \cite{FHSW22} for a fix.

This research was partially funded by the Austrian Science Fund (FWF) under the Esprit Programme [ESP 156]. For the purpose of Open Access, the author has applied a CC BY public copyright licence to any Author Accepted Manuscript (AAM) version arising from this submission.

\section{The deBruijn identity in finite von Neumann algebras}

In this section, we establish the deBruijn identity in finite von Neumann algebras, that is, von Neumann algebras with a faithful normal trace $\tau$. Even in this case, convergence of the entropy production functional $\tau(\L_\ast(\rho)(\log \rho-\log\sigma))$ for general density operators $\rho$, $\sigma$ is problematic, hence we will focus on a class of normal states with good approximation properties for which the entropy production is well-defined.

Let us briefly recall the relevant definitions. Let $H$ be a Hilbert space and $\IB(H)$ the set of bounded linear operators on $H$. A \emph{von Neumann algebra} on $H$ is a subspace $M$ of $\IB(H)$ that contains the identity and is closed under multiplication, taking adjoints and closed in the weak operator topology. A bounded linear functional $\phi\colon M\to \IC$ is called
\begin{itemize}
    \item \emph{normal} if it is continuous with respect to the weak operator topology on the unit ball of $M$,
    \item a \emph{state} if $\phi(x^\ast x)\geq 0$ for all $x\in M$ and $\phi(1)=1$,
    \item \emph{faithful} if $\phi(x^\ast x)>0$ for all $x\in M$, $x\ne 0$,
    \item a (finite) \emph{trace} if $\phi(xy)=\phi(yx)$ for all $x,y\in M$.
\end{itemize}
The set of all normal linear functionals on $M$ is denoted by $M_\ast$.

A closed densely defined operator $x$ on $H$ with polar decomposition $x=u\abs{x}$ is called \emph{affiliated with $M$} if $u\in M$ and all spectral projections of $\abs{x}$ belong to $M$. If $x$ is a positive self-adjoint operator affiliated with $M$ with spectral measure $e_x$ and $\tau$ is a finite normal trace on $M$, then $\tau\circ e_x$ is a measure on the spectrum of $x$ and one can define
\begin{equation*}
    \tau(x)=\int_{[0,\infty)}\lambda\,d(\tau\circ e_x)(\lambda)\in [0,\infty].
\end{equation*}
If $\tau$ is faithful and $1\leq p<\infty$, the noncommutative Lebesgue space $L^p(M,\tau)$ is the space of all closed densely defined operators $x$ affiliated with $M$ such that $\tau(\abs{x}^p)<\infty$. The norm on $L^p(M,\tau)$ is defined by $\norm{x}_p=\tau(\abs{x}^p)^{1/p}$. We write $L^p_+(M,\tau)$ for the cone of positive self-adjoint operators in $L^p(M,\tau)$.

If $M$ admits a faithful normal tracial state, the sum and product of two affiliated operators is closable and the closure is again an affiliated operator. These are called the strong sum and strong product, and we will simply write $x+y$ and $xy$ for the strong operations. The norm on $L^2(M,\tau)$ is induced by an inner product and the map
\begin{equation*}
    \pi_\tau\colon M\to \IB(L^2(M,\tau)),\,\pi_\tau(x)y=xy
\end{equation*}
is a faithful unital $\ast$-representation, called the GNS representation associated with $\tau$. In particular, one can always assume that $H=L^2(M,\tau)$, which we will do in the following.

For $\sigma\in L^1_+(M,\tau)$ and $\alpha\geq 1$ let $B_\alpha(\sigma)=\{\rho\in L^1_+(M,\tau)\mid \alpha^{-1}\sigma\leq \rho\leq \alpha \sigma\}$ and $B(\sigma)=\bigcup_{\alpha\geq 1}B_\alpha(\sigma)$. We say that a self-adjoint operator $x$ on $L^2(M,\tau)$ is \emph{non-singular} if $\1_{\IR\setminus \{0\}}(x)=1$. If $x$ is a non-singular positive self-adjoint operator, then $\log \rho$ can be defined using functional calculus (extending $\log$ to $[0,\infty)$ in an arbitrary way).

%

\begin{lemma}\label{lem:diff_log_bounded}
If $\rho,\sigma\in L^1_+(M,\tau)$ are non-singular and $\rho\in B_\alpha(\sigma)$, then $\abs{\log\rho-\log\sigma}\leq\log\alpha$. In particular, $\log\rho-\log\sigma\in M$.
\end{lemma}
\begin{proof}
As the logarithm is operator monotone, we have
\begin{equation*}
-\log \alpha=\log(\alpha^{-1}\sigma)-\log \sigma\leq \log\rho-\log\sigma\leq \log(\alpha\sigma)-\log\sigma=\log\alpha.\qedhere
\end{equation*}
\end{proof}

\begin{lemma}\label{lem:integral_formula_entropy}
If $x$ is a non-singular operator  positive self-adjoint affiliated with $M$ and $\xi,\eta\in \dom(\log x)$, then
\begin{equation*}
\langle \xi,(\log x)\eta\rangle=\lim_{n\to\infty}\int_{1/n}^n \langle \xi,((1+\lambda)^{-1}-(x+\lambda)^{-1})\eta\rangle\,d\lambda
\end{equation*}
\end{lemma}
\begin{proof}
Note that
\begin{equation*}
\log t=\int_0^\infty\left(\frac 1 {1+\lambda}-\frac 1{t+\lambda}\right)\,d\lambda
\end{equation*}
for $t>0$.

We can assume that $\xi=\eta$; the general case follows by polarization. Let $\mu_\xi$ denote the spectral measure of $x$ with respect to $\xi$. By the spectral theorem,
\begin{align}
\begin{split}\label{eq:quadr_form_log}
\langle\xi,(\log x)\xi\rangle&=\int_0^\infty \log t\,d\mu_\xi(t)\\
&=\int_0^\infty\int_0^\infty\left(\frac 1{1+\lambda}-\frac 1{t+\lambda}\right)\,d\lambda\,d\mu_\xi(t).
\end{split}
\end{align}
Since
\begin{align*}
&\int_0^\infty\int_0^\infty \left\lvert\frac 1{1+\lambda}-\frac 1{t+\lambda}\right\rvert\,d\lambda\,d\mu_\xi(t)\\
&\qquad=\int_0^\infty (\1_{[1,\infty)}(t)\log t-\1_{(0,1)}(t)\log t)\,d\mu_\xi(t)\\
&\qquad=\int_0^\infty\abs{\log t}\,d\mu_\xi(t)\\
&\qquad\leq \left(\int_0^\infty(\log t)^2\,d\mu_\xi(t)\right)^{1/2}\\
&\qquad<\infty,
\end{align*}
we can apply Fubini's theorem to \eqref{eq:quadr_form_log} to obtain
\begin{align*}
\langle\xi,(\log x)\xi\rangle&=\int_0^\infty\int_0^\infty \left(\frac 1{1+\lambda}-\frac 1{t+\lambda}\right)\,d\mu_\xi(t)\,d\lambda\\
&=\int_0^\infty \langle\xi,((1+\lambda)^{-1}-(x+\lambda)^{-1})\xi\rangle\,d\lambda,
\end{align*}
where the last integral converges absolutely. Hence
\begin{equation*}
\langle \xi,(\log x)\xi\rangle=\lim_{n\to\infty}\int_{1/n}^n \langle \xi,((1+\lambda)^{-1}-(x+\lambda)^{-1})\xi\rangle\,d\lambda
\end{equation*}
by the dominated convergence theorem.
\end{proof}

\begin{lemma}\label{lem:approx_entropy}
If $\sigma\in L^1_+(M,\tau)$ is non-singular, $\alpha\geq 1$ and $\rho\in B_\alpha(\sigma)$, then
\begin{equation*}
\left\lVert \int_{1/n}^n ((\sigma+\lambda)^{-1}-(\rho+\lambda)^{-1})\,d\lambda\right\rVert\leq \log\alpha
\end{equation*}
for all $n\in\mathbb N$ and
\begin{equation*}
\lim_{n\to\infty}\int_{1/n}^n ((\sigma+\lambda)^{-1}-(\rho+\lambda)^{-1}))\,d\lambda\to\log\rho-\log\sigma
\end{equation*}
in the weak$^\ast$ topology on $M$.

In particular,
\begin{equation*}
\lim_{n\to\infty}\int_{1/n}^n\tau(\rho ((\sigma+\lambda)^{-1}-(\rho+\lambda)^{-1})))\,d\lambda\to\tau(\rho(\log\rho-\log\sigma)).
\end{equation*}
\end{lemma}
\begin{proof}
First note that since the resolvent is a norm continuous function, the integral 
\begin{equation*}
x_n=\int_{1/n}^n ((\sigma+\lambda)^{-1}-(\rho+\lambda)^{-1})\,d\lambda
\end{equation*}
converges in norm in $M$. Moreover, $\log\rho-\log\sigma$ belongs to $M$ by Lemma \ref{lem:diff_log_bounded}.

Since the inverse is operator monotone decreasing, if $\xi\in \dom(\log \sigma)$, we have
\begin{align*}
\langle\xi,x_n \xi\rangle&\leq \int_{1/n}^n \langle \xi,((\sigma+\lambda)^{-1}-(\alpha\sigma+\lambda)^{-1})\xi\rangle\,d\lambda\\
&\leq \int_0^\infty\langle \xi,((\sigma+\lambda)^{-1}-(\alpha\sigma+\lambda)^{-1})\xi\rangle\,d\lambda\\
&=\langle \xi,(\log(\alpha \sigma)-\log\sigma)\xi\rangle\\
&=(\log\alpha)\norm{\xi}^2,
\end{align*}
where the second-to-last equality follows from Lemma \ref{lem:integral_formula_entropy}. Similarly one can show $\langle\xi,x_n\xi\rangle\geq -(\log \alpha)\norm{\xi}^2$. Thus $\norm{x_n}\leq \log\alpha$.

Moreover, by Lemma \ref{lem:integral_formula_entropy},
\begin{equation*}
\langle\xi,x_n\eta\rangle\to \langle\xi,(\log\rho-\log\sigma)\eta\rangle
\end{equation*}
for $\xi,\eta\in \dom(\log\rho)\cap \dom(\log\sigma)$.

As $\rho$, $\sigma$ are self-adjoint operators affiliated with $M$, so are $\log\rho$ and $\log \sigma$. Thus the intersection $\dom(\log\rho)\cap \dom(\log\sigma)$ is dense in $L^2(M,\tau)$ \cite[Proposition 24]{Ter81}. Therefore, $x_n\to \log\rho-\log \sigma$ weakly, which is the same as weak$^\ast$ convergence on norm bounded sets. In particular,
\begin{equation*}
\langle \rho^{1/2},x_n\rho^{1/2}\rangle\to \langle \rho^{1/2},(\log\rho-\log\sigma)\rho^{1/2}\rangle.\qedhere
\end{equation*}
\end{proof}

\begin{lemma}\label{lem:deriv_approx_entropy}
Let $\sigma\in L^1_+(M,\tau)$ have full support and let $(\rho_t)_{t\geq 0}$ be continuously differentiable curve in $L^1(M,\tau)$ for which there exists $\alpha\geq 1$ such that $\rho_t\in B_\alpha(\sigma)$ for all $t\geq 0$. The function
\begin{equation*}
f_n\colon [0,\infty)\to \IR,\,t\mapsto\int_{1/n}^n \tau(\rho_t((\sigma+\lambda)^{-1}-(\rho_t+\lambda)^{-1}))\,d\lambda
\end{equation*}
is differentiable with derivative
\begin{equation*}
f_n'(t)=\int_{1/n}^n \tau(\dot\rho_t((\sigma+\lambda)^{-1}-(\rho_t+\lambda)^{-1}+\rho_t(\rho_t+\lambda)^{-2}))\,d\lambda.
\end{equation*}
\end{lemma}
\begin{proof}
Let $\alpha\geq 1$ such that $\rho_t\in B_\alpha(\rho)$ for all $t\geq 0$. We have
\begin{equation*}
(\rho_t+\lambda)^{-1}\rho(\rho_t+\lambda)^{-1}\leq \alpha \rho_t(\rho_t+\lambda)^{-2}\leq \frac{\alpha}{\lambda}.
\end{equation*}
As the resolvent is norm continuous, the integrand in the expression for $f_n^\prime(t)$ depends continuously on $\lambda$. Thus the integral is well-defined.

If $s,t\geq 0$, then
\begin{align*}
f_n(t)-f_n(s)&=\int_{1/n}^n \tau((\rho_t-\rho_s)((\sigma+\lambda)^{-1}-(\rho_t+\lambda)^{-1}))\,d\lambda\\
&\quad+\int_{1/n}^n\tau(\rho_s((\rho_t+\lambda)^{-1}-(\rho_s+\lambda)^{-1}))\,d\lambda\\
\begin{split}\label{eq:test}
&=\tau\left((\rho_t-\rho_s)\int_{1/n}^n ((\sigma+\lambda)^{-1}-(\rho_t+\lambda)^{-1})\,d\lambda\right)\\
&\quad+\int_{1/n}^n\tau((\rho_s+\lambda)^{-1}\rho_s(\rho_t+\lambda)^{-1}(\rho_s-\rho_t))\,d\lambda
\end{split}
\end{align*}
For the first integral we have
\begin{equation*}
\int_{1/n}^n ((\sigma+\lambda)^{-1}-(\rho_t+\lambda)^{-1})\,d\lambda\in M
\end{equation*}
and 
\begin{equation*}
\lim_{s\to t}\frac{\rho_t-\rho_s}{t-s}=\dot\rho_t
\end{equation*}
in $L^1(M,\tau)$. Hence
\begin{align*}
&\tau\left(\frac{\rho_t-\rho_s}{t-s}\int_{1/n}^n ((\sigma+\lambda)^{-1}-(\rho_t+\lambda)^{-1})\,d\lambda\right)\\
&\qquad\to \int_{1/n}^n \tau(\dot\rho_t((\sigma+\lambda)^{-1}-(\rho_t+\lambda)^{-1}))\,d\lambda.
\end{align*}

We want to apply the dominated convergence theorem to the second integral. We have
\begin{align*}
\norm{\rho_t-\rho_s}_1\leq \abs{t-s}\sup_{r\in [s,t]}\norm{\dot\rho_r}_1
\end{align*}
and
\begin{align*}
\norm{(\rho_s+\lambda)^{-1}\rho_s(\rho_t+\lambda)^{-1}}&\leq\norm{(\rho_s+\lambda)^{-1}\rho_s}\norm{(\rho_t+\lambda)^{-1}}\leq \frac 1{\lambda}.
\end{align*}
Furthermore, in the limit $s\to t$ we have $(\rho_t-\rho_s)/(t-s)\to \dot\rho_t$ in $L^1(M,\tau)$ and $(\rho_s+\lambda)^{-1}\rho_s(\rho_t+\lambda)^{-1}\to \rho_t(\rho_t+\lambda)^{-2}$ strongly in $M$ by the strong resolvent continuity of the curve $(\rho_s)$. Now the dominated convergence theorem yields the desired derivative.
\end{proof}

\begin{lemma}\label{lem:rel_Hamiltonian_cont}
For $\alpha\geq 1$ and $\sigma\in L^1(M,\tau)$ non-singular, the map
\begin{equation*}
B_\alpha(\sigma)\to M,\,\rho\mapsto \log \rho-\log\sigma
\end{equation*}
is continuous for the $L^1$ norm topology on $B_\alpha(\sigma)$ and the weak$^\ast$ topology on $M$.
\end{lemma}
\begin{proof}
Since $\norm{\log\rho-\log\sigma}\leq \log \alpha$ for $\rho\in B_\alpha(\sigma)$, by \Cref{lem:dense_pos_cone_Lp} it suffices to show that whenever $\norm{\rho-\rho_i}_1\to 0$, then $\tau((\log\rho_i-\log \rho)\nu)\to 0$ for all $\nu\in B(\sigma)$. We have
\begin{align*}
\tau((\log\rho_i-\log\rho)\nu)=\tau((\log \rho_i-\log \nu)\nu)-\tau((\log\rho-\log\nu)\nu),
\end{align*}
and the claimed convergence follows from \cite[Theorem 3.7 (2)]{Ara78}.
\end{proof}

\begin{proposition}\label{prop:deriv_entropy_semifinite}
Let $\sigma\in L^1_+(M,\tau)$ be non-singular and let $(\rho_t)_{t\geq 0}$ be a continuously differentiable curve in $L^1(M,\tau)$ with $\tau(\rho_t)=\tau(\rho_0)$ for all $t\geq 0$ and for which there exists $\alpha\geq 1$ such that $\rho_t\in B_\alpha(\sigma)$ for all $t\geq 0$.

The map
\begin{equation*}
f\colon[0,\infty)\to \IR,\,t\mapsto \tau(\rho_t(\log\rho_t-\log \sigma)
\end{equation*}
is differentiable with derivative
\begin{equation*}
f'(t)=\tau(\dot\rho_t(\log\rho_t-\log\sigma)).
\end{equation*}
\end{proposition}
\begin{proof}
Let
\begin{equation*}
f_n\colon [0,\infty)\to \IR,\,t\mapsto\int_{1/n}^n \tau(\rho_t((\sigma+\lambda)^{-1}-(\rho_t+\lambda)^{-1}))\,d\lambda.
\end{equation*}
By Lemma \ref{lem:approx_entropy} the sequence $(f_n)$ converges to $f$ pointwise. By Lemma \ref{lem:deriv_approx_entropy}, $f_n$ is differentiable with derivative 
\begin{equation*}
f_n'(s)=\int_{1/n}^n \tau(\dot\rho_s((\sigma+\lambda)^{-1}-(\rho_s+\lambda)^{-1}+\rho_s(\rho_s+\lambda)^{-2}))\,d\lambda.
\end{equation*}
Let
\begin{align*}
g_n(s)&=\int_{1/n}^n \tau(\dot\rho_s((\sigma+\lambda)^{-1}-(\rho_s+\lambda)^{-1}))\,d\lambda,\\
h_n(s)&=\int_{1/n}^n \tau(\dot\rho_s\rho_s(\rho_s+\lambda)^{-2})\,d\lambda.
\end{align*}
By Lemma \ref{lem:approx_entropy} we have
\begin{equation*}
\abs{g_n(s)}\leq \norm{\dot\rho_s}_1\log\alpha
\end{equation*}
and $g_n(s)\to \tau(\dot\rho_s(\log\rho_s-\log\sigma))$ as $n\to\infty$.

Since $s\mapsto \dot\rho_s$ is continuous in $L^1(M,\tau)$ and $s\mapsto \log\rho_s-\log\sigma$ is bounded and weak$^\ast$ continuous by Lemma \ref{lem:rel_Hamiltonian_cont}, the map $s\mapsto \tau(\dot\rho_s(\log\rho_s-\log\sigma))$ is continuous.

For $h_n$ note that if $\xi\in L^2(M,\tau)$ and $\mu_\xi$ denotes the spectral measure of $\rho_s$ with respect to $\xi$, then
\begin{equation*}
\int_{1/n}^n\langle\xi,\rho_s(\rho_s+\lambda)^{-2}\xi\rangle\,d\lambda\nearrow \int_0^\infty \int_0^\infty \frac{x}{(x+\lambda)^2}\,d\lambda\,d\mu_\xi(x)=\norm{\xi}_2^2.
\end{equation*}
Thus
\begin{equation*}
\abs{h_n(s)}\leq \norm{\dot\rho_s}_1
\end{equation*}

Moreover, letting $\xi_{\pm}=(\dot\rho_s)_\pm^{1/2}$ we obtain
\begin{align*}
h_n(s)&=\int_{1/n}^n \langle\xi_+,\rho_s(\rho_s+\lambda)^{-2}\xi_+\rangle\,d\lambda-\int_{1/n}^n \langle\xi_-,\rho_s(\rho_s+\lambda)^{-2}\xi_-\rangle\,ds\\
&\to \norm{\xi_+}_2^2-\norm{\xi_-}_2^2.
\end{align*}
As $\tau(\rho_t)=\tau(\rho)$ for all $t\geq 0$ by assumption, we have
\begin{equation*}
\norm{\xi_+}_2^2-\norm{\xi_-}_2^2=\tau((\dot\rho_s)_+-(\dot\rho_s)_-)=\frac{d}{ds}\tau(\rho_s)=0.
\end{equation*}
Hence $g_n(s)+h_n(s)\to \tau(\dot\rho_s(\log\rho_s-\log\sigma))$ as $n\to\infty$ and the sequence $(g_n+h_n)$ is uniformly bounded on compact subsets of $[0,\infty)$.

Therefore,
\begin{align*}
f_n(t)-f_n(s)&=\int_s^t (g_n(r)+h_n(r))\,dr\\
&\to \int_s^t \tau(\dot\rho_r(\log\rho_r-\log\sigma))\,dr
\end{align*}
by the dominated convergence theorem. Since $f_n(t)\to f(t)$, $f_n(s)\to f(s)$, it follows that $f$ is differentiable with the desired derivative.
\end{proof}

As a consequence we obtain the deBruijn identity for quantum Markov semigroups on finite von Neumann algebras. Recall that a \emph{quantum Markov semigroup} is a semigroup $(P_t)$ of normal unital completely positive maps on $M$ such that $t\mapsto P_t(x)$ is weak$^\ast$ continuous for all $x\in M$. As $P_t$ is normal, it is the adjoint of a bounded linear operator $P_{t\ast}$ on $M_\ast$, and $(P_{t\ast})$ is a strongly continuous semigroup on $M_\ast$. The \emph{generator} $\L_\ast$ of $(P_{t\ast})$ is defined by
\begin{align*}
    \dom(\L_\ast)&=\{\rho\in L^1(M,\tau)\mid \lim_{t\to 0}\frac 1 t(\rho-P_{t\ast}(\rho))\text{ exists}\}\\
    \L_\ast(\rho)&=\lim_{t\to 0}\frac{\rho-P_{t\ast}(\rho)}{t}.
\end{align*}
By the semigroup property, if $\rho\in \dom(\L_\ast)$, then $P_{t\ast}\rho\in \dom(\L_\ast)$ for all $t\geq 0$ and $\frac{d}{dt}P_{t\ast}(\rho)=-\L_\ast(P_{t\ast}(\rho))$.

\begin{theorem}[De Bruijn's identity -- finite case]\label{thm:DeBruijn_semifinite}
Let $(P_t)$ be a quantum Markov semigroup on $M$ and let $\sigma\in L^1_+(M,\tau)$ be non-singular and satisfy $P_{t\ast}(\sigma)=\sigma$ for all $t\geq 0$.

For $\rho\in B(\sigma)\cap \dom(\L_\ast)$ and $t\geq 0$ one has
\begin{equation*}
\frac{d}{dt}\tau(P_{t\ast}(\rho)(\log P_{t\ast}(\rho)-\log\sigma))=-\tau(\L_\ast(P_{t\ast}(\rho))(\log P_{t\ast}(\rho)-\log\sigma)).
\end{equation*}
\end{theorem}
\begin{proof}
By the definition of the generator, the curve $(P_{t\ast}(\rho))_{t\geq 0}$ is continuously differentiable with derivative $\frac{d}{dt}P_{t\ast}(\rho)=-\L_\ast(P_{t\ast}(\rho))$. Moreover, if $\alpha\geq 1$ such that $\rho\in B_\alpha(\sigma)$, then
\begin{equation*}
\alpha^{-1}\sigma=\alpha^{-1}P_{t\ast}(\sigma)\leq P_{t\ast}(\rho)\leq \alpha P_{t\ast}(\sigma)=\alpha\sigma
\end{equation*}
for all $t\geq 0$ since $P_t$ is positive. In other words, $P_{t\ast}(\rho)\in B_\alpha(\sigma)$ for all $t\geq 0$. Finally,
\begin{equation*}
\tau(P_{t\ast}(\rho))=\tau(P_t(1)\rho)=\tau(\rho)
\end{equation*}
for all $t\geq 0$.

Thus Proposition \ref{prop:deriv_entropy_semifinite} yields the claim.
\end{proof}

\section{Haagerup reduction and main result}

In this section, we set up the technical apparatus to reduce the general case of the deBruijn identity to the finite case treated in the previous section.

Let $M$ be a von Neumann algebra, $\phi$ a normal faithful state on $M$ and $(P_t)$ a quantum Markov semigroup on $M$. Let $\L_\ast$ denote the generator of the strongly continuous semigroup $(P_{t\ast})$ on $M_\ast$ and let $\L=(\L_\ast)^\ast$.

We use Haagerup $L^p$ spaces. We do not give a full definition and refer to \cite{Ter81} for an exposition. Let us just remark that $L^p(M)$ consists of closed densely defined operators affiliated with the crossed product $M\rtimes_{\sigma^\phi}\IR$ and there is an order-preserving isometric isomorphism $M_\ast\to L^1(M)$, $\psi\mapsto h_\psi$. With a slight abuse of notation, we also write $P_{t\ast}$ and $\L_\ast$ for the semigroup and generator on $L^1(M)$ instead of $M_\ast$.

Similar to the tracially symmetric case, for $\alpha\geq 1$ we define 
\begin{equation*}
B_{\alpha}(\phi)=\{\psi\in M_\ast^+\mid \alpha^{-1}\phi\leq \psi\leq \alpha\phi\}
\end{equation*}
and $B(\phi)=\bigcup_{\alpha\geq 1}B_{\alpha}(\phi)$.

\begin{lemma}
The set $B(\phi)\cap \dom(\L_\ast)$ is dense in $M_\ast^+$.
\end{lemma}
\begin{proof}
By Lemma \ref{lem:dense_pos_cone_Lp} the set $B(\phi)$ is dense in $M_\ast^+$. If $\psi\in B(\phi)$, then $n(\L_\ast+n)^{-1}(\psi)\in B(\phi)\cap \dom(\L_\ast)$ since $n(\L+n)^{-1}$ is a normal unital completely positive map, and $n(\L_\ast+n)^{-1}(\psi)\to \psi$.
\end{proof}

For $\psi\in M_\ast^+$ the support $s(\psi)$ of $\psi$ is the smallest projection $p\in M$ such that $\psi(1-p)=0$. If $\psi,\omega\in M_\ast^+$, let $\Delta_{\psi,\omega}$ denote the relative modular operator and $E_{\psi,\omega}$ its spectral measure. The \emph{(Umegaki) relative entropy} of $\psi$ with respect to $\omega$ is defined as
\begin{equation*}
D(\psi\Vert\omega)=\int_{(0,\infty)}\lambda \log\lambda\,d\langle h_\omega^{1/2},E_{\psi,\omega}(\lambda) h_\omega^{1/2}\rangle+\infty\cdot\psi(1-s(\omega))\in (-\infty,\infty].
\end{equation*}
Here we use the usual convention $\infty\cdot 0=0$. For more information on the relative entropy in von Neumann algebras, we refer to \cite{Ara76,Ara78,Hia18}.

If $\psi\in B_\alpha(\phi)$ for some $\alpha\geq 1$, the operator $h=\log h_\psi-\log h_\phi$ has norm bounded by $\log\alpha$ and belongs to $M$ (see \cite[Theorem B.1]{Hia18}). It satisfies 
\begin{equation*}
D(\omega\Vert \phi)=D(\omega\Vert \psi)+\omega(h)
\end{equation*}
for $\omega\in M_\ast^+$. In particular,
\begin{equation*}
D(\psi\Vert \phi)=\psi(h)=\tr(h_\psi(\log h_\psi-\log h_\phi))<\infty.
\end{equation*}

With these properties, the proof of the following lemma is the same as in the finite case (Lemma \ref{lem:rel_Hamiltonian_cont}).

\begin{lemma}\label{lem:rel_Hamiltonian_cont_II}
For $\alpha\geq 1$ the map
\begin{equation*}
B_\alpha(\phi)\to M,\,\psi\mapsto \log h_\psi-\log h_\phi
\end{equation*}
is continuous for the $M_\ast$ norm topology on $B_\alpha(\phi)$ and the weak$^\ast$ topology on $M$.
\end{lemma}

To use the results from the previous section in the current setting when $M$ is not necessarily finite, we will take advantage of Haagerup's reduction method. For that purpose, the following two approximation lemmas will be useful.

\begin{lemma}\label{lem:entropy_extension}
Let $M\subset \tilde M$ be a unital inclusion of von Neumann algebras with faithful normal conditional expectation $E\colon \tilde M\to M$.
\begin{itemize}
\item[(a)] If $\psi\in M_\ast^+$, then $D(\psi\circ E\Vert \phi\circ E)=D(\psi\Vert \phi)$.

\item[(b)]If $\psi\in B(\phi)$, then $E(\log h_{\psi\circ E}-\log h_{\phi\circ E})=\log h_\psi-\log h_\phi$.
\end{itemize}
\end{lemma}
\begin{proof}
(a) Let $\iota\colon M\to N$ denote the inclusion map. By two applications of the data processing inequality (see \cite[Theorem 4.1]{Hia18}), we obtain
\begin{equation*}
D(\psi\Vert\phi)=D(\psi\circ E\circ\iota\Vert\psi\circ E\circ\iota)\leq D(\psi\circ E\Vert \phi\circ E)\leq D(\psi\Vert\phi).
\end{equation*}
(b) 
If $\omega\in M_\ast^+$, then
\begin{equation*}
D(\omega\circ E\Vert \psi\circ E)+\omega(E(\log h_{\psi\circ E}-\log h_{\phi\circ E}))=D(\omega\circ E\Vert\phi\circ E).
\end{equation*}
By the previous part, if $\omega\in B(\phi)$, then
\begin{equation*}
\omega(E(\log h_{\psi\circ E}-\log h_{\phi\circ E}))=D(\omega\Vert \phi)-D(\omega\Vert \psi)=\omega(\log h_\psi-\log h_\phi).
\end{equation*}
As $B(\phi)$ is dense in $M_\ast^+$, it follows that 
\begin{equation*}
E(\log h_{\psi\circ E}-\log h_{\phi\circ E})=\log h_\psi-\log h_\phi.\qedhere
\end{equation*}
\end{proof}

\begin{lemma}\label{lem:martingale_conv_entropy}
Let $(M_n)$ be an increasing sequence of unital von Neumann subalgebras of $M$ such that the union $\bigcup_n M_n$ is weak$^\ast$ dense in $M$.
\begin{itemize}
\item[(a)] If $\psi\in M_\ast^+$, then
\begin{equation*}
D(\psi|_{M_n}\Vert \phi|_{M_n})\nearrow D(\psi\Vert \phi).
\end{equation*}
\item[(b)] If $\psi\in B(\phi)$, then
\begin{equation*}
\log h_{\psi_n}-\log h_{\phi_n}\to \log h_\phi-\log h_\psi
\end{equation*}
in the weak$^\ast$ topology. Here $\psi_n$ (resp. $\phi_n$) denotes the restriction of $\psi$ (resp. $\phi)$ to $M_n$.
\end{itemize}
\end{lemma}
\begin{proof}
(a) This is known as continuity of the relative entropy under martingale convergence, see \cite[Theorem 4.1]{Hia18}.

(b) Note that if $\psi\in B_{\alpha}(\phi)$, then $\alpha^{-1} \phi_n\leq \psi_n\leq \alpha \phi_n$ for all $n\in\mathbb N$. In other words, $\psi_n\in B_\alpha(\phi_n)$ for all $n\in\mathbb N$.

By the first part, if $\omega\in B(\phi)$, then
\begin{align*}
\omega(\log h_{\psi_n}-\log h_{\phi_n})&=\omega|_{M_n}(\log h_{\psi_n}-\log h_{\phi_n})\\
&=D(\omega|_{M_n}\Vert \phi|_{M_n})-D(\omega|_{M_n}\Vert \psi|_{M_n})\\
&\to D(\omega\Vert \phi)-D(\omega\Vert\psi)\\
&=\omega(\log h_\psi-\log h_\phi).
\end{align*}
Since $B(\phi)$ is dense in $M_\ast^+$ and $\norm{\log h_{\psi_n}-\log h_{\phi_n}}\leq \log\alpha$ for all $n\in\mathbb N$ if $\psi\in B_\alpha(\phi)$, it follows that $\log h_{\psi_n}-\log h_{\phi_n}\to \log h_{\psi}-\log h_\phi$ in the weak$^\ast$ topology.
\end{proof}

Let us now briefly recall Haagerup's reduction method. A detailed account can be found in \cite{HJX10}.

Let $G=\bigcup_{n\in\IN}2^{-n}\IZ$ viewed as a discrete group. The modular group $\sigma^\phi$ restricts to an action of $G$ on $M$ and we can form the crossed product $M\rtimes_{\sigma^\phi}G$, which we denote by $\tilde M$. Since $G$ is discrete, the dual weight $\tilde\phi$ of $\phi$ is a state and there exists a (unique) normal faithful conditional expectation $E\colon \tilde M\to M$ such that $\tilde\phi=\phi\circ E$. Moreover, there exists an increasing sequence $(M_n)$ of finite von Neumann subalgebras of $\tilde M$ such that $\bigcup_n M_n$ is weak$^\ast$ dense in $\tilde M$. For the purposes of the present section, the construction of the algebra $M_n$ is irrelevant, but we will give some details later in Section \ref{sec:MLSI_intertwining}.

\begin{proposition}\label{prop:deriv_entropy}
Let $(\psi_t)_{t\geq 0}$ be a continuously differentiable curve in $M_\ast$ such that $\psi_t(1)=\psi_0(1)$ for all $t\geq 0$ and for which there exists $\alpha\geq 1$ such that $\psi_t\in B_\alpha(\phi)$ for all $t\geq 0$.

The map
\begin{equation*}
f\colon [0,\infty)\to \IR,\,t\mapsto D(\psi_t\Vert \phi)
\end{equation*}
is differentiable with derivative
\begin{equation*}
f'(t)=\dot\psi_t(\log h_{\psi_t}-\log h_\phi).
\end{equation*}
\end{proposition}
\begin{proof}
We keep the notation from the previous discussion. For $\omega\in M_\ast$ let $\tilde\omega=\omega\circ E$, which is consistent with the notation for the dual weight.

In particular, $(\tilde \psi_t)$ is a continuously differentiable curve in $\tilde M_\ast$ such that $\tilde \psi_t(1)=\tilde\psi_0(1)$ and $\tilde\psi_t\in B_\alpha(\tilde\phi)$ for all $t\geq 0$. Moreover, $\frac{d}{dt}\tilde\psi_t=\dot\psi_t\circ E$.

By Lemma \ref{lem:entropy_extension} we have $D(\tilde \psi_t\Vert \tilde \phi)=D(\psi_t\Vert \phi)$ and 
\begin{equation*}
\dot\psi_t(\log h_{\psi_t}-\log h_\phi)=\left(\frac{d}{dt}\tilde\psi_t\right)(\log h_{\tilde\psi_t}-\log h_{\tilde\phi}).
\end{equation*}
Further, let $\omega^n=\tilde\omega|_{M_n}$ for $\omega\in M_\ast$. Again, $(\psi^n_t)_{t\geq 0}$ is a continuously differentiable curve in $(M_n)_\ast$ such that $\psi^n_t(1)=\psi^n_0(1)$ and $\psi^n_t\in B_\alpha(\phi^n)$ for all $t\geq 0$ and $\dot\psi^n_t=\dot\psi_t\circ E|_{M_n}$.

By Lemma \ref{lem:martingale_conv_entropy}, $D(\psi^n_t\Vert \phi^n)\nearrow D(\tilde\psi_t\Vert \tilde\phi)=D(\psi_t\Vert \phi)$ and
\begin{align*}
\dot\psi^n_t(\log h_{\psi^n_t}-\log h_{\phi^n})&=\dot \psi_t(E(\log h_{\psi^n_t}-\log h_{\phi^n}))\\
&\to \dot\psi_t(E(\log h_{\tilde \psi_t}-\log h_{\tilde\phi}))\\
&=\dot\psi_t(\log h_{\psi_t}-\log h_\phi).
\end{align*}
Furthermore,
\begin{equation*}
\abs{\dot\psi^n_t(\log h_{\psi^n_t}-\log h_{\phi^n})}\leq \norm{\dot\psi_t}_1\log \alpha.
\end{equation*}
Since $M_n$ is finite, we can apply Proposition \ref{prop:deriv_entropy_semifinite} and get
\begin{equation*}
D(\psi^n_t\Vert \phi)-D(\psi^n_s\Vert \phi)=\int_s^t \dot \psi^n_r(\log h_{\psi^n_r}-\log h_{\phi^n})\,dr.
\end{equation*}
By the previous discussion, the left side converges to $D(\psi_t\Vert \phi)-D(\psi_s\Vert\phi)$, while the right side converges to $\int_s^t \dot\psi_r(\log h_{\psi_r}-\log h_\phi)\,dr$ by the dominated convergence theorem. As $r\mapsto \dot\psi_r(\log h_{\psi_r}-\log h_\phi)\,dr$ is continuous by Lemma \ref{lem:rel_Hamiltonian_cont_II}, this implies the differentiability of $f$ with the desired derivative.
\end{proof}

Now let $(P_t)$ be a quantum Markov semigroup on $M$ and let $\L_\ast$ denote the generator of $(P_{t\ast})$. For $\psi\in B(\phi)\cap \dom(\L_\ast)$ we define the \emph{entropy production} of $\psi$ with respect to $\phi$ as
\begin{equation*}
\I_\L(\psi\Vert \phi)=\L_\ast(\psi)(\log h_\psi-\log h_\phi)=\tr(\L_\ast(h_\psi)(\log h_\psi-\log h_\phi)).
\end{equation*}
The entropy production for open quantum systems was first studied by Spohn. We refer to his article \cite{Spo78} for the physical interpretation of this quantity.

\begin{theorem}[De Bruijn's identity -- $\sigma$-finite case]\label{thm:DeBruijn_general}
Let $M$ be a von Neumann algebra, $\phi$ a faithful normal state on $M$ and $(P_t)$ a quantum Markov semigroup on $M$ such that $P_{t\ast}(\phi)=\phi$ for all $t\geq 0$.

For $\psi\in B(\phi)\cap \dom(\L_\ast)$ and $t\geq 0$, one has
\begin{equation*}
\frac{d}{dt}D(P_{t\ast}(\psi)\Vert \phi)=-\I_\L(P_{t\ast}(\psi)\Vert \phi).
\end{equation*}
\end{theorem}
\begin{proof}
For $t\geq 0$ let $\psi_t=P_{t\ast}(\psi)$. By definition of the generator, the curve $(\psi_t)$ is continuously differentiable in $M_\ast$. Moreover, $\frac{d}{dt}\psi_t(1)=\psi_t(\L(1))=0$ for all $t\geq 0$. Finally, 
\begin{equation*}
\alpha^{-1}\phi=\alpha^{-1}P_{t\ast}(\phi)\leq P_{t\ast}(\psi)\leq \alpha P_{t\ast}(\phi)=\alpha\phi
\end{equation*}
for all $t\geq 0$ since $P_t$ is positivity-preserving and $\phi$ is invariant under $P_t$.

Hence the claim follows from Proposition \ref{prop:deriv_entropy}.
\end{proof}

If it exists, an invariant state of a quantum Markov semigroup is not necessarily unique. To get decay bounds of the relative entropy in this case, one needs to choose the reference state depending on the input state. For this purpose, let us introduce the conditional expectation onto the fixed-point algebra.

Let $(P_t)$ be a quantum Markov semigroup such that $P_{t\ast}(\phi)=\phi$ for all $t\geq 0$. The fixed-point algebra $N$ of $(P_t)$ is given by 
\begin{equation*}
N=\{x\in M\mid P_t(x)=x\text{ for all }t\geq 0\}.
\end{equation*}
As $\phi$ is faithful, an application of the Kadison--Schwarz inequality shows that $N$ is a von Neumann subalgebra of $M$ (see \cite[Section 2]{Wat79}). Moreover, for every $x\in M$ the strong limit
\begin{equation*}
E(x)=\lim_{t\to\infty}\frac 1 t\int_0^t P_s(x)\,ds
\end{equation*}
exists, belongs to $N$, and $E$ defines a $\phi$-preserving normal faithful conditional expectation of $M$ onto $N$. The predual of $E$ is given by
\begin{equation*}
E_\ast(\omega)=\lim_{t\to\infty}\frac 1 t\int_0^t P_{s\ast}(\omega)\,ds
\end{equation*}
for $\omega\in M_\ast$, where the limit exists in the norm topology, and $E_\ast$ is a norm-one projection onto $\{\omega\in M_\ast\mid P_{t\ast}(\omega)=\omega\text{ for all }t\geq 0\}$. See \cite[Theorem A]{Wat79} for a proof of these facts.

\begin{theorem}\label{thm:main}
Let $M$ be a von Neumann algebra, $\phi$ a faithful normal state on $M$ and $(P_t)$ a quantum Markov semigroup such that $P_{t\ast}(\phi)=\phi$ for all $t\geq 0$. Let $E$ denote the $\phi$-preserving faithful normal conditional expectation onto the fixed-point algebra of $(P_t)$. For $\beta>0$, the following assertions are equivalent:
\begin{itemize}
\item[(i)] $\beta D(\psi\Vert E_\ast(\psi))\leq \I_\L(\psi\Vert E_\ast(\psi))$ for all $\psi\in \dom(\L_\ast)$ with $E_\ast(\psi)\in B(\psi)$ faithful,
\item[(ii)]$D(P_{t\ast}(\psi)\Vert E_\ast(\psi))\leq e^{-\beta t}D(\psi\Vert E_\ast(\psi))$ for all $\psi\in \dom(\L_\ast)$ with $E_\ast(\psi)\in B(\psi)$ faithful, $t\geq 0$,
\item[(iii)]$D(P_{t\ast}(\psi)\Vert E_\ast(\psi))\leq e^{-\beta t}D(\psi\Vert E_\ast(\psi))$ for all $\psi\in M_\ast^+$, $t\geq 0$.
\end{itemize}
\end{theorem}
\begin{proof}
Since $E_\ast(P_{t\ast}(\psi))=E_\ast(\psi)$ for all $t\geq 0$, the implication  (i) $\implies$ (ii) now follows from Theorem \ref{thm:DeBruijn_general} with $\phi=E_\ast(\psi)$ by an application of Grönwall's lemma.

For (ii)$\implies$(i), note that $\I_\L(\psi\Vert E_\ast(\psi))=-\frac{d}{dt}|_{t=0}D(P_{t\ast}(\psi)\Vert E_\ast(\psi))$ by \Cref{thm:DeBruijn_general}. By (ii),
\begin{align*}
    \I_\L(\psi\Vert E_\ast(\psi))&=\lim_{t\to 0}\frac 1 t(D(\psi\Vert E_\ast(\psi))-D(P_{t\ast}(\psi)\Vert E_\ast(\psi)))\\
    &\geq \lim_{t\to 0}\frac {1-e^{-\beta t}}{t}D(\psi\Vert E_\ast(\psi))\\
    &=\beta D(\psi\Vert E_\ast(\psi)),
\end{align*}
which settles (i).

Clearly (iii) implies (ii). To prove that (ii) implies (iii), we proceed by a number of reductions. First, if $\psi\in M_\ast^+$, let $\psi_{n,t}=n P_{t\ast}(\L_\ast+n)^{-1}(\psi)$. Since $n(\L+n)^{-1}$ is a normal unital completely positive map that commutes with $P_t$, we have by the data processing inequality and lower semicontinuity of the entropy
\begin{align*}
D(P_{t\ast}(\psi)\Vert E_\ast(\psi))&\leq \liminf_{n\to\infty}D(\psi_{n,t}\Vert E_\ast(\psi))\\
&=\liminf_{n\to\infty}D(n(\L_\ast+n)^{-1}(P_{t\ast}\psi)\Vert n(\L_\ast+n)^{-1}(E_\ast(\psi)))\\
&\leq\limsup_{n\to\infty}D(n(\L_\ast+n)^{-1}(P_{t\ast}\psi)\Vert n(\L_\ast+n)^{-1}(E_\ast(\psi)))\\
&\leq D(P_{t\ast}(\psi)\Vert E_\ast(\psi))
\end{align*}
for all $t\geq 0$. Thus (iii) holds for $\psi\in M_\ast^+$ with $E_\ast(\psi)\in B(\psi)$ faithful.

If we only assume that $E_\ast(\psi)$ is faithful and $\psi\leq \alpha E_\ast(\psi)$ for some $\alpha\geq 1$, then $\psi_\epsilon=(1-\epsilon)\psi+\epsilon E_\ast(\psi)$ satisfies
\begin{equation*}
D(P_{t\ast}(\psi_\epsilon)\Vert E_\ast(\psi_\epsilon))\to D(P_{t\ast}(\psi)\Vert E_\ast(\psi))
\end{equation*}
as $\epsilon\to 0$ for all $t\geq 0$ by the convexity and lower semicontinuity of $D$. Thus (iii) also holds in this case.

If only $\psi\in M_\ast^+$ with $E_\ast(\psi)$ faithful, then by \cite[Lemma 5, Theorem 5]{FHSW22} there exists sequences $(\psi_n)$ in $M_\ast^+$ and $(\alpha_n)$ in $[1,\infty)$ such that $\psi_n\leq \alpha_n E_\ast(\psi)$ and $\psi_n\to \psi$ in norm, $D(\psi_n\Vert E_\ast(\psi))\to D(\psi\Vert E_\ast(\psi))$. By the chain rule for the relative entropy \cite[Theorem 2]{Pet91},
\begin{equation*}
    D(\psi_n\Vert E_\ast(\psi))=D(\psi_n\Vert E_\ast(\psi_n))+D(E_\ast(\psi_n)\Vert E_\ast(\psi))\geq D(\psi_n\Vert E_\ast(\psi_n)).
\end{equation*}
Thus $\limsup_{n\to\infty}D(\psi_n\Vert E_\ast(\psi_n))\leq D(\psi\Vert E_\ast(\psi))$. It follows from the previous steps and the lower semicontinuity of the relative entropy that
\begin{align*}
    D(P_{t\ast}(\psi)\Vert E_\ast(\psi))
    &\leq \liminf_{n\to\infty}D(P_{t\ast}(\psi_n)\Vert E_\ast(\psi_n))\\
    &\leq e^{-\beta t}\limsup_{n\to\infty}D(\psi_n\Vert E_\ast(\psi_n))\\
    &\leq e^{-\beta t}D(\psi\Vert E_\ast(\psi)).
\end{align*}
Finally, to drop the assumption that $E_\ast(\psi)$ is faithful, we can consider $\psi_\epsilon=(1-\epsilon)\psi+\epsilon\phi$ and use convexity and lower semicontinuity of the relative entropy as above.
\end{proof}

\begin{remark}
    It is customary to only consider normal states instead of general positive normal linear functionals in the modified logarithmic Sobolev inequality. However, since both $\psi\colon D(\psi\Vert E_\ast(\psi))$ and $\psi\mapsto \I_\L(\psi\Vert E_\ast(\psi))$ are $1$-homogeneous, all the equivalent statements from the previous theorem are equivalent to the corresponding statements for states.
\end{remark}

\begin{remark}
For tracially symmetric quantum Markov semigroups, the entropy production is usually defined as $\I_\L(\rho)=\tau(\L_\ast(\rho)\log\rho)$ without reference to any invariant state. A similar definition for non-tracial states is problematic because $\psi\log\psi$ is not in $L^1(M)$.

However, up to technical restrictions which reference states are permissible in the definition of $\I_\L$, the entropy production in this setting is still independent of the choice of invariant state. Indeed, if $\phi_1$, $\phi_2$ are normal faithful states invariant under $(P_t)$ and $\phi_2\in B(\phi_1)$, then 
\begin{equation*}
E(\log h_{\phi_2}-\log h_{\phi_1})=E(\log h_{\phi_1\circ E}-\log h_{\phi_2\circ E})=\log h_{\phi_1}-\log h_{\phi_2}
\end{equation*}
by a similar argument as in the proof of Lemma \ref{lem:entropy_extension}. In other words, $\log h_{\phi_1}-\log h_{\phi_2}\in \ker(\L)$.

Thus, if $\psi\in B(\phi_1)=B(\phi_2)$, then
\begin{align*}
\I_\L(\psi\Vert \phi_1)-\I_\L(\psi\Vert \phi_1)=\L_\ast(\psi)(\log h_{\phi_2}-\log h_{\phi_1})=0.
\end{align*}
In the context of KMS-symmetric quantum Markov semigroups on matrix algebras, the same observation was made in \cite[Lemma 2]{Capel21}.
\end{remark}

\begin{definition}[Modified logarithmic Sobolev inequality]
Let $(P_t)$ be a quantum Markov semigroup and $\beta>0$. We say that $(P_t)$ satisfies the \emph{modified logarithmic Sobolev inequality} $\mathrm{MLSI}(\beta)$ if one of the following equivalent properties of the previous theorem holds.

We say that $(P_t)$ satisfies the \emph{complete modified logarithmic Sobolev inequality} $\mathrm{CMLSI}(\beta)$ if $(P_t\otimes \id_N)$ satisfies $\mathrm{MLSI}(\beta)$ for every $\sigma$-finite von Neumann algebra $N$.
\end{definition}

\begin{corollary}
    If the quantum Markov semigroup $(P_t)$ has a faithful normal invariant state and satisfies $\mathrm{MLSI}(\beta)$ for some $\beta>0$, then $P_{t\ast}(\psi)\to E_\ast(\psi)$ in norm as $t\to \infty$ for all $\psi\in M_\ast$.
\end{corollary}
\begin{proof}
    First let $\psi\in M_\ast^+$ with $\psi(1)=1$ and $D(\psi\Vert E_\ast(\psi))<\infty$. By \Cref{thm:main}, we have $D(P_{t\ast}(\psi)\Vert E_\ast(\psi))\leq e^{-\beta t}D(\psi\Vert E_\ast(\psi))$ for all $t\geq 0$. By the quantum Pinsker inequality \cite[Theorem 3.1]{HOT81},
    \begin{equation*}
        \norm{P_{t\ast}(\psi)-E_\ast(\psi)}\leq \sqrt 2D(P_{t\ast}(\psi)\Vert E_\ast(\psi))\leq \sqrt 2 e^{-\beta t}D(\psi\Vert E_\ast(\psi))\to 0
    \end{equation*}
    as $t\to \infty$.

    By homogeneity, the same is true without the condition $\psi(1)=1$. Finally, since $B(\phi)\subset \{\psi\in M_\ast^+\mid D(\psi\Vert E(\psi))<\infty\}$ is dense in $M_\ast^+$ by \Cref{lem:dense_pos_cone_Lp}, the claim for arbitrary $\psi\in M_\ast$ follows by approximation and linearity.
\end{proof}

\section{Modified logarithmic Sobolev inequality through intertwining}\label{sec:MLSI_intertwining}

So far we have not assumed any kind of detailed balance condition on the semigroup for the characterization of relative entropy decay. Detailed balance however can be very useful to obtain bounds on the exponential decay rate (see \cite{KT13,CM15} for example). In this section we illustrate how the intertwining method developed for GNS-symmetric semigroups on finite-dimensional von Neumann algebras \cite{CM17,CM20,MWZ24} and tracially symmetric semigroups on finite von Neumann algebras \cite{WZ21} (see also \cite{BGJ22,LJL24} for a closely related approach) can be extended to GNS-symmetric quantum Markov semigroups on arbitrary von Neumann algebras. 

We keep the setup from the previous section: $M$ is a von Neumann algebra and $\phi$ is a faithful normal state on $M$. A quantum Markov semigroup $(P_t)$ is \emph{GNS-symmetric with respect to $\phi$} if 
\begin{equation*}
    \phi(P_t(x)^\ast y)=\phi(x^\ast P_t(y))
\end{equation*}
for all $x,y\in M$ and $t\geq 0$.

Note that by \cite[Lemma 4.16]{GJLL25}, if a quantum Markov semigroup is GNS-symmetric with respect to a faithful normal state, it is also GNS-symmetric with respect to any other invariant state.

One advantage of GNS symmetry is that the semigroup on $M$ gives rise to a strongly continuous symmetric contraction semigroup $(T_t)$ on $L^2(M)$, called the \emph{GNS implementation} of $(P_t)$. On the dense subset $\{xh_\phi^{1/2}\mid x\in M\}$, it is defined by $T_t(xh_\phi^{1/2})=P_t(x) h_\phi^{1/2}$. By the spectral theorem, $T_t$ converges strongly to to the orthogonal projection onto its fixed-point space as $t\to\infty$. As a consequence, $P_t(x)\to E(x)$ in the strong operator topology as $t\to\infty$, where $E$ is the conditional expectation onto the fixed-point algebra of $(P_t)$. This has the following consequence for the long-time behavior of the predual semigroup.

\begin{lemma}\label{lem:long_time_GNS}
    If $(P_t)$ is a GNS-symmetric quantum Markov semigroup and $E$ is the $\phi$-preserving conditional expectation onto the fixed-point algebra, then $P_{t\ast}(\psi)\to E_\ast(\psi)$ as $t\to\infty$ for all $\psi\in M_\ast$.
\end{lemma}
\begin{proof}
    Since $\phi$ is faithful, the set $\{x\phi\mid x\in M\}$ is dense in $M_\ast$. By GNS symmetry,
    \begin{equation*}
        P_{t\ast}(x\phi)(y)=(x\phi)(P_t(y))=\phi(P_t(y)x)=\phi(yP_t(x))=(P_t(x)\phi)(y)
    \end{equation*}
    for all $x,y\in M$. Thus $P_{t\ast}(x\phi)=P_t(x)\phi$.

    As discussed above, $P_t(x)h_\phi^{1/2}\to E(x)h_\phi^{1/2}$ as $t\to\infty$. By Hölder's inequality, $P_t(x)h_\phi\to E(x)h_\phi$, which translates to $P_t(x)\phi\to E(x)\phi$ under the natural isometric isomorphism $L^1(M)\cong M_\ast$. Finally, a similar argument as above shows that $E(x)\phi=E_\ast(x\phi)$. Thus $P_{t\ast}(x\phi)\to E_\ast(x\phi)$ as $t\to\infty$. Since $\norm{P_{t\ast}}=1$, the same convergence holds for all $\psi\in M_\ast$.
\end{proof}

\begin{lemma}\label{lem:FM_implies_MLSI}
Let $(P_t)$ be a GNS-symmetric quantum Markov semigroup with generator $\L$ and let $\beta>0$. If
\begin{equation*}
\I_\L(P_{t\ast}(\psi)\Vert E_\ast(\psi))\leq e^{-\beta t}\I_\L(\psi\Vert E_\ast(\psi))
\end{equation*}
for all $\dom(\L_\ast)$ with $E_\ast(\psi)\in B(\psi)$ faithful and $t\geq 0$, then $(P_t)$ satisfies $\mathrm{MLSI}(\beta)$.
\end{lemma}
\begin{proof}
Let $\psi\in\dom(\L_\ast)$ with $E_\ast(\psi)\in B(\psi)$ faithful. By \Cref{lem:long_time_GNS}, we have $P_{t\ast}(\psi)\to E_\ast(\psi)$ in norm as $t\to\infty$. Moreover $P_{t\ast}(\psi)\in B_\alpha(E_\ast(\psi))$. It follows from \cite[Theorem 3.7]{Ara78} that 
\begin{equation*}
D(P_{t\ast}(\psi)\Vert E_\ast(\psi))\to D(E_\ast(\psi)\Vert E_\ast(\psi))=0.
\end{equation*}
If we combine this with \Cref{thm:DeBruijn_general}, we get
\begin{align*}
D(\psi\Vert E_\ast(\psi))&=\lim_{T\to\infty}\int_0^T \I_\L(P_{t\ast}(\psi)\Vert E_\ast(\psi))\,dt\\
&\leq \I_\L(\psi\Vert E_\ast(\psi))\int_0^\infty e^{-\beta t}\,dt\\
&=\beta^{-1}\I_\L(\psi\Vert E_\ast(\psi)).\qedhere
\end{align*}
\end{proof}

\begin{definition}[Fisher monotonicity]
Let $(P_t)$ be a GNS-symmetric quantum Markov semigroup with generator $\L$ and $\beta\in\IR$. We say that $(P_t)$ satisfies the Fisher monotonicity bound $\mathrm{FM}(\beta)$ if 
\begin{equation*}
\I_\L(P_{t\ast}(\psi)\Vert E_\ast(\psi))\leq e^{-\beta t}\I_\L(\psi\Vert E_\ast(\psi))
\end{equation*}
for all $\psi\in\dom(\L_\ast)$ with $E_\ast\in B(\psi)$ faithful and $t\geq 0$.

We say that $(P_t)$ satisfies the complete Fisher monotonicity bound $\mathrm{CFM}(\beta)$ if $(P_t\otimes\id_N)$ satisfies $\mathrm{FM}(\beta)$ for every $\sigma$-finite von Neumann algebra $N$.
\end{definition}

\begin{remark}
With this definition, the previous lemma can be reformulated as saying that $\mathrm{FM}(\beta)$ implies $\mathrm{MLSI}(\beta)$ for $\beta>0$ (and of course $\mathrm{CFM}(\beta)$ also implies $\mathrm{CMLSI}(\beta)$ for $\beta>0$). But unlike the modified logarithmic Sobolev inequality, $\mathrm{FM}(\beta)$ is also non-trivial for $\beta\leq 0$.

Moreover, $\mathrm{FM}(\beta)$ for $\beta>0$ implies not only that $t\mapsto D(P_{t\ast}(\psi)\Vert E_\ast(\psi))$ decays exponentially with rate $\beta$, but gives additional information on the second derivative of this function. In particular, this function is convex. This behavior was called \emph{convex entropy decay} in \cite{CDP09}. As noted there \cite[Remark 2.2]{CDP09}, it is in general stronger than $\mathrm{MLSI}(\beta)$.
\end{remark}

To state the intertwining criterion, we need the first-order differential calculus for GNS-symmetric quantum Markov semigroups developed in \cite{Wir22b,Wir24}. If $(P_t)$ is a GNS-symmetric quantum Markov semigroup and $(T_t)$ its GNS implementation on $L^2(M)$, define a quadratic form $\E$ on $L^2(M)$ by
\begin{align*}
    \dom(\E)&=\{a\in L^2(M)\mid \lim_{t\to 0}\frac 1 t\langle a,a-T_t(a)\rangle_2\text{ exists}\},\\
    \E(a)&=\lim_{t\to 0}\frac 1 t \langle a,a-T_t(a)\rangle_2.
\end{align*}
This form is called the \emph{modular quantum Dirichlet form} associated with $(P_t)$. The quadratic forms that arise in this way from GNS-symmetric quantum Markov semigroups can be characterized intrinsically, see \cite[Theorem 4.11]{Cip97}, \cite[Theorem 5.7]{GL95} in combination with \cite[Section 3]{Wir22b}.

Let
\begin{equation*}
    \mathfrak A_\E=\{xh_\phi^{1/2}\in \bigcap_{z\in \IC}\dom(\Delta_\phi^{iz})\mid \forall z\in\IC\,\exists y_z\in M\colon \Delta_\phi^{iz}(x h_\phi^{1/2})=y_z h_\phi^{1/2}\in\dom(\E)\}.
\end{equation*}
In \cite[Theorem 6.3]{Wir22b} it was shown that $\mathfrak A_\E$ is a core for $\E$ and if $xh_\phi^{1/2},h_\phi^{1/2}y\in \mathfrak A_\E$, then $xh_\phi^{1/2}y\in \mathfrak A_\E$. Let $A$ be the norm closure of $\{x\in M\mid xh_\phi^{1/2}\in \mathfrak A_\E\}$.

By \cite[Theorem 6.8]{Wir22b}, there exists a quadruple $(\H,\J,(\U_t),\delta)$ consisting of a Hilbert bimodule $\H$ over $A$, an anti-unitary involution $\J\colon \H\to\H$, a strongly continuous unitary group $(\U_t)$ on $\H$ and a closable operator $\delta\colon\mathfrak A_\E\to \H$ satisfying
\begin{itemize}
    \item $\J(x\xi y)=y^\ast (\J\xi)x^\ast$ for all $x,y\in A$, $\xi\in \H$
    \item $\U_t(x\xi y)=\sigma^\phi_t(x)\xi\sigma^\phi_t(y)$ for all $x,y\in A$, $\xi\in \H$, $t\in\IR$,
    \item $\U_t\J=\J\U_t$ for all $t\in\IR$,
    \item $\delta\circ\Delta_\phi^{it}=\U_t\circ\delta$ for all $t\in\IR$,
    \item $\delta\circ J_\phi=\J\circ\delta$,
    \item $\delta(xh_\phi^{1/2}y)=x\delta(h_\phi^{1/2}y)+\delta(xh_\phi^{1/2})y$ whenever $xh_\phi^{1/2},h_\phi^{1/2}y\in \mathfrak A_\E$
\end{itemize}
such that
\begin{equation*}
    \E(a)=\norm{\delta(a)}_\H^2
\end{equation*}
for all $a\in \mathfrak A_\E$.

Under the minimality condition that $\H$ is the closure of the $A$-bimodule generated by $\H$, the quadruple $(\H,\J,(\U_t),\delta)$ is essentially uniquely determined by $(P_t)$, up to a unitary bimodule isomorphism that intertwines the operators $\J$, $(\U_t)$ and $\delta$, and it is called the \emph{first-order differential calculus} associated with $(P_t)$.

The set $A$ is a weak$^\ast$ dense unital $C^\ast$-subalgebra of $M$. If the left and right actions of $A$ on $\H$ can be extended to normal maps on $M$, the quantum Markov semigroup $(P_t)$ is called \emph{$\Gamma$-regular}. A characterization of $\Gamma$-regularity in terms of the carré du champ operator $\Gamma$ is given in \cite[Theorem 7.2]{Wir22b}.

To prove the next theorem, we will again rely on Haagerup's reduction method. This time we need some more details of the construction. Let $G=\bigcup_{n\in\IN}2^{-n}\IZ$, $\tilde M=M\rtimes_{\sigma^\phi} G\subset \IB(\ell^2(G))\overline{\otimes} M$, $\tilde\phi$ the dual weight of $\phi$ on $M$ and $E\colon \tilde M\to M$ the unique normal faithful conditional expectation such that $\tilde\phi=\phi\circ E$.

A sequence $(M_n)$ of finite von Neumann subalgebras of $M$ such that $\bigcup_n M_n$ is weak$^\ast$ dense in $M$ is constructed as follows. If $x\in \lambda(G)\subset M\rtimes_{\sigma^\phi}G$, then $x\in \mathcal Z(M_{\tilde \phi})$, the center of the centralizer of $\tilde\phi$. Moreover, $x\tilde\phi$ restricts to a trace on the centralizer $M_{x\phi}$. As shown in \cite[Section 2]{HJX10}, one can choose a sequence $(a_n)$ in $\lambda(G)$ such that $\bigcup_n M_{e^{-a_n}\tilde\phi}$ is weak$^\ast$ dense in $M$. Then one can take $M_n=M_{e^{-a_n}\tilde\phi}$. We write $\tau_n$ for the restriction of $e^{-a_n}\tilde\phi$ to $M_n$. Note that in contrast to the notation used in \cite{HJX10}, we do not normalize $\tau_n$.

\begin{theorem}\label{thm:intertwining}
Let $(P_t)$ be a $\Gamma$-regular GNS-symmetric quantum Markov semigroup, $(T_t)$ the GNS implementation of $(P_t)$ on $L^2(M)$, and $(\H,\J,(\U_t),\delta)$ the associated first-order differential calculus. Write $\pi_l^\H$ and $\pi_r^\H$ for the left and right action of $M$ on $\H$. 

Let $K\in\mathbb R$. If there exists a family $(\vec T_t)$ of bounded linear operators on $\H$ such that
\begin{itemize}
\item $\delta T_t\supset\vec T_t\delta$ for all $t\geq 0$,
\item $\vec T_t^\ast \pi_l^\H(\cdot)\vec T_t\leq_{\cp}e^{-2Kt} \pi_l^\H\circ P_t$ for all $t\geq 0$,
\item $\vec T_t^\ast \pi_r^\H(\cdot)\vec T_t\leq_{\cp}e^{-2Kt} \pi_r^\H\circ P_t$ for all $t\geq 0$,
\end{itemize}
then $(P_t)$ satisfies $\mathrm{CFM}(2K)$.

In particular, if $K>0$, then $(P_t)$ satisfies $\mathrm{CMLSI}(2K)$.
\end{theorem}

\begin{remark}
    In this theorem and its proof, we use the notation $\Phi\leq_{\mathrm{cp}}\Psi$ to express that $\Psi-\Phi$ is completely positive.
\end{remark}

\begin{proof}
We use the notation from the previous discussion. By \cite[Theorem 4.1]{HJX10}, $\id_{\ell^2(G)}\otimes P_t$ restricts to a completely positive map $\tilde P_t$ on $\tilde M$, and $(\tilde P_t)$ is a quantum Markov semigroup, which is GNS-symmetric with respect to $\tilde \phi$. Let $\tilde \E$ denote the associated quantum Dirichlet form.

The $L^2$ space of the crossed product $\tilde M$ is canonically identified with $\ell^2(G)\otimes L^2(M)$ and the GNS implementation of $(\tilde P_t)$ under this identification is $(1_{\ell^2(G)}\otimes T_t)$. Let $\tilde\E$ be the associated quantum Dirichlet form.

Let $\tilde \H=\ell^2(G)\otimes \H$ with $\pi_l^{\tilde \H}$, $\pi_r^{\tilde \H}$ the restriction of $\id_{\IB(\ell^2(G))}\otimes \pi_l^\H$ and $\id_{\IB(\ell^2(G)}\otimes\pi_r^\H$ to $\tilde M$, let $\tilde\J=J\otimes \J$ with $J$ the involution on $\ell^2(G)$ given by $J\1_g=\1_{-g}$, let $\tilde\U_t=1_{\ell^2(G)}\otimes\U_t$, and let $\tilde \delta$ be the restriction of $1_{\ell^2(G)}\otimes \bar\delta$ to $\mathfrak A_{\tilde \E}$. A direct computation shows that $(\tilde\H,\tilde\J,(\tilde \U_t),\tilde\delta)$ satisfies the defining properties of a first-order differential calculus for $(\tilde P_t)$, except for the minimality condition, which can be achieved by replacing $\tilde \H$ by the closed bimodule generated by the range of $\tilde\delta$.

Since $T_t$ is continuous for the graph norm of $\delta$, the intertwining relation $\delta T_t\supset \vec T_t\delta$ can be extended to $\bar\delta T_t\supset\vec T_t\bar\delta$. Thus
\begin{equation*}
    (1\otimes\vec T_t)\tilde\delta\subset (1\otimes \vec T_t)(1\otimes\bar\delta)\subset (1\otimes\bar\delta)(1\otimes T_t).
\end{equation*}
Since $\mathfrak A_{\tilde \E}$ is invariant under $1\otimes T_t$, the restriction of the right side to $\mathfrak A_{\tilde \E}$ is $\tilde\delta(1\otimes T_t)$. Hence
\begin{equation*}
    (1\otimes \vec T_t)\tilde\delta\subset\tilde\delta(1\otimes T_t).
\end{equation*}
Furthermore,
\begin{align*}
    (1\otimes\vec T_t)^\ast \pi_l^{\tilde \H}(\cdot)(1\otimes\vec T_t)&=(\id_{\IB(\ell^2(G))}\otimes \vec T_t^\ast\pi_l^\H(\cdot)\vec T_t)|_{\tilde M}\\
    &\leq_{\mathrm{cp}} e^{-2Kt}(\id_{\IB(\ell^2(G))}\otimes \pi_l^\H\circ P_t)|_{\tilde M}\\
    &=e^{-2Kt}\pi_l^{\tilde \H}\circ\tilde P_t.
\end{align*}
The proof of $(1\otimes\vec T_t)^\ast \pi_r^{\tilde \H}(\cdot)(1\otimes\vec T_t)\leq_{\mathrm{cp}} e^{-2Kt}\pi_r^{\tilde \H}\circ\tilde P_t$ is analogous.

Moreover, as $L(G)\otimes\IC 1$ is pointwise invariant under $\tilde P_t$, it follows that $\tilde P_t$ is a bimodule map over $L(G)\otimes \IC 1$.

By \cite[Lemma 4.5]{GJLL25}, the semigroup $(\tilde P_t)$ restricts to a $\tau_n$-symmetric quantum Markov semigroup $(P_t^n)$ on $M_n$. Note that the tracial $L^2$ space $L^2(M_n,\tau_n)$ can be isometrically identified with a subspace of $L^2(\tilde M)$ via the map $x\mapsto x e^{-a_n/2} h_{\tilde \phi}^{1/2}$ for $x\in M_n$. Let $\E_n$ denote the quantum Dirichlet form associated with $(P_t^n)$ on $L^2(M_n,\tau_n)$. If $x\in M_n$, then
\begin{align*}
\E_n( x)&=\sup_{t\geq 0}\frac 1 t \tau_n(x^\ast(x-P_t^n(x)))\\
&=\sup_{t\geq 0}\frac 1 t\tilde\phi((xe^{-a_n/2})^\ast(xe^{-a_n/2}-P_t^n(x)e^{-a_n/2}))\\
&=\sup_{t\geq 0}\frac 1 t\tilde\phi((xe^{-a_n/2})^\ast(xe^{-a_n/2}-\tilde P_t(xe^{-a_n/2})))\\
&=\tilde \E(x e^{-a_n/2}h_{\tilde\phi}^{1/2}),
\end{align*}
where we used that $\tilde P_t$ is a bimodule map over $L(G)\otimes\IC 1$ for the third equality.

In particular,
\begin{equation*}
\dom(\E_n)\cap M_n=\{x\in M_n\mid xe^{-a_n/2}h_{\tilde\phi}^{1/2}\in \dom(\tilde \E)\}.
\end{equation*}
Note that on $M_n$, the modular group of $\tilde\phi$ acts as $\sigma^{\tilde\phi}_t(x)=e^{i a_n t}x e^{-i a_n t}$. Thus every element of $M_n$ is entire analytic for $\sigma^{\tilde\phi}$. Hence
\begin{equation*}
\dom(\E_n)\cap M_n=\{x\in M_n\mid xe^{-a_n/2}h_{\tilde\phi}^{1/2}\in\AA_{\tilde\E}\}
\end{equation*}
and
\begin{equation*}
\E_n(x)=\norm{\tilde\delta(x e^{-a_n/2} h_{\tilde\phi}^{1/2})}_{\tilde\H}^2
\end{equation*}
for $x\in \dom(\E_n)\cap M_n$.

Let
\begin{equation*}
\delta_n\colon \dom(\E_n)\cap M_n\to\tilde\H,\,\delta_n(x)=\tilde\delta(x e^{-a_n/2}h_{\tilde\phi}^{1/2}).
\end{equation*}
Note that since $a_n$ is in the centralizer of $\tilde \phi$ and $M_n$ is the centralizer of $e^{-a_n}\tilde \phi$, we have 
\begin{equation*}
xe^{-a_n/2}h_{\tilde\phi}^{1/2}=e^{-a_n/2}h_{\tilde\phi}^{1/2}x=h_{\tilde\phi}^{1/2}e^{-a_n/2}x
\end{equation*}
for all $x\in M_n$.

Since $\tilde \E(e^{a_n/2}h_{\tilde\phi}^{1/2})=0$, we have 
\begin{align*}
\delta_n(xy)&=\tilde\delta(xy e^{-a_n/2}h_{\tilde\phi}^{1/2})\\
&=\tilde\delta(xe^{-a_n/2}h_{\tilde\phi}^{1/2}\cdot e^{a_n/2}h_{\tilde\phi}^{1/2}\cdot y e^{-a_n/2}h_{\tilde\phi}^{1/2})\\
&=\delta_n( x)e^{a_n/2}e^{-a_n/2}y+x e^{-a_n/2}e^{a_n/2}\delta_n( y)\\
&=\delta_n( x)y+x\delta_n(y).
\end{align*}
Moreover,
\begin{align*}
\tilde \J\delta_n( x)&=\tilde \J\tilde\delta(xe^{-a_n/2}h_{\tilde\phi}^{1/2})\\
&=\tilde\delta(h_{\tilde\phi}^{1/2}e^{-a_n/2}x^\ast)\\
&=\tilde\delta(x^\ast e^{-a_n/2}h_{\tilde\phi}^{1/2})\\
&=\delta_n( x^\ast).
\end{align*}
Therefore, the first-order differential calculus associated with $\E_n$ is a corestriction of $(\tilde\H,\tilde\J,(\id),\delta_n)$ and $\E_n$ is $\Gamma$-regular.

By restriction,
\begin{itemize}
\item $\delta_n T_t^n\supset (1\otimes\vec T_t)\delta_n$ for all $t\geq 0$,
\item $(1\otimes\vec T_t)^\ast \pi_l^{\tilde\H}(\cdot)(1\otimes\vec T_t)\leq_{\cp}e^{-2Kt} \pi_l^{\tilde \H}\circ P_t^n$ for all $t\geq 0$,
\item $(1\otimes\vec T_t^\ast) \pi_r^{\tilde \H}(\cdot)(1\otimes\vec T_t)\leq_{\cp}e^{-2Kt} \pi_r^{\tilde \H}\circ P_t^n$ for all $t\geq 0$,
\end{itemize}
where $(T_t^n)$ is the GNS implementation of $(P_t^n)$ on $L^2(M_n,\tau_n)$.

Now we are in the position to apply the results for tracially symmetric quantum Markov semigroups. By \cite[Corollary 2.9, Theorem 3.1]{WZ21}, the semigroup $(P_t^n)$ satisfies $\mathrm{FM}(2K)$. By the approximation Lemmas \ref{lem:entropy_extension} and \ref{lem:martingale_conv_entropy}, this implies $\mathrm{FM}(2K)$ for $(P_t)$. Furthermore, as the assumptions of the theorem continue to hold if one replaces $(P_t)$ by $(P_t\otimes\id_N)$ for a $\sigma$-finite von Neumann algebra $N$, we get $\mathrm{CFM}(2K)$. Finally, the claim about $\mathrm{CMLSI}(2K)$ follows from Lemma \ref{lem:FM_implies_MLSI}.
\end{proof}

\begin{remark}
    To extend the assumptions on $(P_t)$ to the semigroup $(\tilde P_t)$ on the crossed product, it is crucial that the inequalities in the assumption hold for the completely positive order, not only the positive order.

    In the tracial case or in finite dimensions, the implication from the intertwining criterion to Fisher monotonicity typically passes through a gradient estimate (also known as entropic curvature bound), which implies Fisher monotonicity via a gradient flow argument for a noncommutative Wasserstein metric on the state space \cite{CM17,CM20,WZ21}. In the infinite-dimensional GNS-symmetric case, both the formulation of the gradient estimate and the construction of the noncommutative Wasserstein distance are challenging open questions.
\end{remark}

\begin{example}
    Let $G$ be a countable (discrete) group and let $\pi\colon G\to O(H)$ be an orthogonal representation of $G$. A map $b\colon G\to H$ is called a $1$-cocycle if $b(gh)=b(g)+\pi(g)b(h)$ for all $g,h\in G$. We assume that there exists an orthonormal basis $(e_i)_{i\in I}$ of $H$ such that $\langle b(g^{-1})|e_i\rangle\in \{0,1\}$ for all $g\in G$ and $i\in I$. Such cocycles exist for example for free groups \cite[Proof of Lemma 1.2]{Haa79} and Coxeter groups \cite[Section 7.3]{Boz87}. In each case, $\norm{b(g)}^2$ is the word length of $g$ with respect to the standard generating set.

    Let $v_i=\sum_{g\in G}\langle b(g^{-1})|e_i\rangle\ket g\bra g\in \IB(\ell^2(G))$ for $i\in I$.  Define a quadratic form on the space $\mathrm{HS}(\ell^2(G))$ of Hilbert--Schmidt operators on $\ell^2(G)$ by
    \begin{align*}
        \E_i\colon \mathrm{HS}(\ell^2(G))\to [0,\infty),\,\E_i(x)&=\tr(\abs{v_i x-xv_i}^2).
    \end{align*}
    Clearly, $\E_i$ is bounded for every $i\in I$.

    Let $\sigma=\sum_{g\in G}w(g)\ket g\bra g$ for a weight $w\colon G\to (0,\infty)$ such that $\sum_{g\in G}w(g)=1$, and let $\phi=\tr(\,\cdot\,\sigma)$ on $\IB(\ell^2(G))$. By \cite[Example 5.1]{Wir24}, the form $\E_i$ is a modular quantum Dirichlet form (with respect to $\phi$) for every $i\in I$. Let $(P_t)$ denote the associated quantum Markov semigroup and $(T_t^i)$ its GNS implementation. A direct calculation shows that
    \begin{equation*}
        \E_i(y)=\tr(y^\ast(y-v_i y v_i-(1-v_i)y(1-v_i))
    \end{equation*}
    and
    \begin{equation*}
        T_t^i(y)=e^{-t}y+(1-e^{-t})(v_iy v_i+(1-v_i)y(1-v_i)).
    \end{equation*}

    Define quadratic forms on $\mathrm{HS}(\ell^2(G))$ by
    \begin{align*}
        \dom(\E)&=\left\{y\in \mathrm{HS}(\ell^2(G)): \sum_{i\in I}\E_i(y)<\infty\right\},\,\E(y)=\sum_{i\in I}\E_i(y),\\
        \dom(\E_{\neg i})&=\left\{y\in \mathrm{HS}(\ell^2(G)):\sum_{j\in I\setminus\{i\}}\E_j(y)<\infty \right\},\,\E_{\neg i}(y)=\sum_{j\in I\setminus\{i\}}\E_j(y)
    \end{align*}
    As a sum of bounded quadratic forms, $\E$ and $\E_{\neg i}$ are closed. Moreover, since linear combinations of the operators $\ket g\bra h$ with $g,h\in G$ belong to $\dom(\E)\subset \dom(\E_{\neg i})$, the forms $\E$ and $\E_{\neg i}$ are densely defined. It follows from the characterization in \cite[Theorem 4.11]{Cip97} that $\E$ and $\E_{\neg i}$ are  modular quantum Dirichlet forms. Let $(P_t)$ and $(P_t^{\neg i})$ denote the associated GNS-symmetric quantum Markov semigroups on $\IB(\ell^2(G))$ and $(T_t)$, $(T_t^{\neg i})$ their GNS implementation on $\mathrm{HS}(\ell^2(G))$.

    In fact, if $y\in \dom(\E)$ and $g,h\in G$, then
    \begin{align*}
        \E(y,\ket g\bra h)&=\sum_{i\in I}\tr((v_i y-yv_i)^\ast\langle b(g^{-1})-b(h^{-1})|e_i\rangle\ket g\bra h)\\
        &=\sum_{i\in I}\langle b(g^{-1})-b(h^{-1})|e_i\rangle\tr(y^\ast(v_i\ket g\bra h-\ket g\bra h v_i))\\
        &=\sum_{i\in I}\langle b(g^{-1})-b(h^{-1})|e_i\rangle^2\tr(y^\ast \ket g\bra h)\\
        &=\norm{b(g^{-1})-b(h^{-1})}^2\tr(y^\ast \ket g\bra h).
    \end{align*}

    Thus $\ket g\bra h$ belongs to the domain of the generator $\L_2$ of $\E$ and $\L_2(\ket g\bra h)=\norm{b(g^{-1})-b(h^{-1})}^2\ket g\bra h$. In particular, the linear hull of $\{\ket g\bra h : g,h\in G\}$ is a core for $\L_2^n$ for all $n\in\IN$. A similar calculation shows that $\ket g\bra h$ is also an eigenvector of $T_t^i$ and $T_t^{\neg i}$. Since they have a common orthonormal basis of eigenvectors, the operators $T_t^i$ and $T_t^{\neg i}$ commute, and a direct calculation on the eigenbasis shows $T_t=T_t^i T_t^{\neg i}$.

    The first-order differential calculus for $\E$ can be described as follows: Let $H^\IC$ denote the complexification of $H$ and $\zeta\mapsto\bar\zeta$ the complex conjugation on $H^\IC$. Up to isomorphism, $\H=\mathrm{HS}(\ell^2(G))\otimes H^{\mathbb C}$ with the bimodule actions on the first tensor factor, $\delta(y)=\sum_{i\in I}[v_i,y]\otimes e_i$, $\J(y\otimes \zeta)=y^\ast\otimes\bar\zeta$ and $\U_t(y\otimes\zeta)=\sigma^{it}y\sigma^{-it}\otimes\zeta$. In particular, $(P_t)$ is $\Gamma$-regular.

    Since $\ket g\bra h$ is also an eigenvector of the operator $[v_i,\cdot\,]$ on $\mathrm{HS}(\ell^2(G))$ for all $i\in I$, we get $[v_i,T_t^i(y)]=T_t^i([v_i,y])$ for all $y\in \mathrm{HS}(\ell^2(G))$, and the same is true if $T_t^i$ is replaced by $T_t^{\neg i}$. In particular, if $E_i$ denotes the orthogonal projection onto the kernel of $[v_i,\cdot\,]$, then $T_t^i$ and $T_t^{\neg i}$ commute with $E_i$.

    Define $\vec T_t$ on $\mathrm{HS}(\ell^2(G))\otimes H$ by
    \begin{equation*}
        \vec T_t\left(\sum_{i\in I}y_i\otimes e_i\right)=\sum_{i\in I}T_t^{\neg i}(T_t^i(1-E_i)+e^{-2t}E_i)(y_i))\otimes e_i.
    \end{equation*}
    In particular, if $y\in\dom(\E)$, then
    \begin{align*}
        \vec T_t(\bar\delta(y))&=\sum_{i\in I}T_t^{\neg i}(T_t^i(1-E_i)+e^{-2t}E_i)([v_i,y]))\otimes e_i\\
        &=\sum_{i\in I}T_t([v_i,y])\otimes e_i=\sum_{i\in I}[v_i,T_t(y))]\otimes e_i\\
        &=\bar\delta(T_t(y)).
    \end{align*}
    By the Kadison--Schwarz inequality and GNS symmetry,
    \begin{equation*}
        \tr(\abs{T_t^{\neg i}(y)}^2x)\leq \tr(\abs{y}^2P_t^{\neg i}(x)).
    \end{equation*}
    Moreover,
    \begin{equation*}
        \tr(\abs{(T_t^i(1-E_i)+e^{-2t}E_i)(y)}^2x)\leq e^{-2t}\tr(\abs{y}^2P_t^i(x))
    \end{equation*}
    can be shown as in the proof of \cite[Corollary 3.17]{MWZ24}. Altogether, we have proved that
    \begin{equation*}
        \vec T_t^\ast \pi_r^\H(\cdot)\vec T_t\leq e^{-2t}\pi_r^\H\circ P_t.
    \end{equation*}
    The proof for $\pi_l^\H$ instead of $\pi_r^\H$ is analogous.

    To get bounds in the completely positive order, note that we only used that the operators $v_i$ are commuting projections. The quantum Dirichlet form associated with $(P_t\otimes \id)$ has the same form as $\E$ with $v_i$ replaced by $v_i\otimes 1$, which are again commuting projections. Hence the same argument applied to $(P_t\otimes \id)$ instead of $(P_t)$ gives the desired complete bounds.

    We have shown that $(P_t)$ satisfies the assumptions of \Cref{thm:intertwining}, hence it satisfies the Fisher monotonicity estimate $\mathrm{CFM}(2)$ and the modified logarithmic Sobolev inequality $\mathrm{CMLSI}(2)$.

    Note that if $g,s,t\in G$, then
    \begin{align*}
        \tr(\L_2(\ket t\bra s)\lambda_g \sigma^{1/2})&=\norm{b(s^{-1})-b(t^{-1})}^2 w(t)^{1/2}\delta_{s,gt}\\
        &=\norm{b(t^{-1})+\pi(t^{-1})b(g^{-1})-b(t^{-1})}^2\tr(\ket t\bra s\lambda_g \sigma^{1/2})\\
        &=\norm{b(g)}^2\tr(\ket t\bra s\lambda_g \sigma^{1/2}).
    \end{align*}
    Thus $\lambda_g\sigma^{1/2}\in \dom(\L_2)$ and $\L_2(\lambda_g \sigma^{1/2})=\norm{b(g)}^2\lambda_g \sigma^{1/2}$. Thus $T_t(\lambda_g \sigma^{1/2})=e^{-t\norm{b(g)}^2}\lambda_g\sigma^{1/2}$ and $P_t(\lambda_g)=e^{-t\norm{b(g)}^2}\lambda_g$. In particular, $(P_t)$ leaves $L(G)$ invariant.

    It is an interesting open question whether the restriction of $(P_t)$ to $L(G)$ satisfies $\mathrm{MLSI}(2)$. In the case when $G$ is a free group and $b$ the cocycle induced by the standard length function, this is true (see \cite[Example 5.16]{WZ21}, \cite[Corollary 5.8]{BGJ23}), but the general case, including the cocycle on Coxeter groups mentioned above, seems to be open.
\end{example}

\appendix
\section{Positive cones in noncommutative \texorpdfstring{$L^p$}{Lp} spaces}

The following result is well-known. Since we could not find a precise reference, we decided to include the short proof for the reader's convenience.

\begin{lemma}\label{lem:dense_pos_cone_Lp}
If $M$ is a von Neumann algebra and $\phi$ is a faithful normal state on $M$, then $B(\phi)=\{\psi\in M_\ast^+\mid\exists \alpha\geq 1\colon \alpha^{-1}\phi\leq \psi\leq\alpha \phi\}$ is dense in $M_\ast^+$.
\end{lemma}
\begin{proof}
Since $\phi$ is faithful, $h_\phi^{1/2}$ is a cyclic vector for $M$ on $L^2(M)$. Thus for every $\psi\in M_\ast^+$ there exists a sequence $(x_n)$ in $M$ such that $x_n h_\phi^{1/2}\to h_\psi^{1/2}$ in $L^2(M)$. Let $\psi_n\in M_\ast^+$ such that $h_{\psi_n}=h_\phi^{1/2}(x_n^\ast x_n+\frac 1 n)h_\phi^{1/2}$.

By Hölder's inequality, $h_{\psi_n}\to h_\psi$ in $L^1(M)$. Moreover, $\frac 1n h_\phi\leq h_{\psi_n}\leq (\norm{x_n}^2+\frac 1 n)h_\phi$. Thus $\psi_n\in B(\phi)$ and $\psi_n\to \psi$ in norm in $M_\ast$.
\end{proof}


\bibliographystyle{alpha}
\bibliography{references_LSI_v2}



\end{document}